\documentclass{amsart}
 \usepackage{graphicx}
 \newtheorem{thm}{Theorem}[section]
 \newtheorem{cor}[thm]{Corollary}
 \newtheorem{lema}[thm]{Lemma}
 
 \theoremstyle{definition}
 \newtheorem{definition}[thm]{Definition}

 \newtheorem{rem}[thm]{Remark}

\newcommand{\R}{\mathbb R}
\newcommand{\N}{\mathbb N}
\newcommand{\lam}{\lambda}
\newcommand{\ve}{\varepsilon}

\begin{document}

\markboth{J.P. Pinasco and A.M. Salort} {Curves in the Fucik  spectrum}

%
 
%

\title{ASYMPTOTIC BEHAVIOR OF THE CURVES \\ IN THE FUCIK  SPECTRUM}

\author[J P Pinasco]{Juan Pablo Pinasco}
\address{Departamento de Matem\'atica
 \hfill\break \indent FCEN - Universidad de Buenos Aires and
 \hfill\break \indent   IMAS - CONICET.
\hfill\break \indent Ciudad Universitaria, Pabell\'on I \hfill\break \indent   (1428)
Av. Cantilo s/n. \hfill\break \indent Buenos Aires, Argentina.}
\email{jpinasco@dm.uba.ar}
\urladdr{http://mate.dm.uba.ar/~jpinasco}

\author[A Salort]{Ariel Martin Salort}
\address{Departamento de Matem\'atica
 \hfill\break \indent FCEN - Universidad de Buenos Aires and
 \hfill\break \indent   IMAS - CONICET.
\hfill\break \indent Ciudad Universitaria, Pabell\'on I \hfill\break \indent   (1428)
Av. Cantilo s/n. \hfill\break \indent Buenos Aires, Argentina.}
\email{asalort@dm.uba.ar}
\urladdr{http://mate.dm.uba.ar/~asalort}



\maketitle

\begin{abstract}
In this work we study the asymptotic behavior of the curves of the Fu{\v{c}}{\'{\i}}k spectrum
for weighted second order linear ordinary differential equations. We prove a Weyl type
asymptotic behavior of the hyperbolic type curves in the spectrum in terms of some
integrals of the weights. We present an algorithm which computes the intersection of the
Fu{\v{c}}{\'{\i}}k spectrum with rays through the origin, and we compare their values with
the asymptotic ones.
\end{abstract}

\section{Introduction}
We are interested here in the shape of the curves in the Fu{\v{c}}{\'{\i}}k  spectrum
 for the following problem
\begin{align} \label{ecu}
\begin{cases}
   -u''= \alpha \, m(x) u^+- \beta \, n(x) u^- \qquad x\in (0,L)\\
   u(0)=u(L)=0,
\end{cases}
\end{align}
where $(\alpha,\beta) \in \R^2_+$, and the functions $m$, $n\in C[0,L]$ are positive and
bounded below by some positive constant $c$. As usual, given a function $u$ we denote
$u^\pm = \max\{0,\pm u\}$ the positive and negative parts of $u$.

This problem was introduced in the '70s by Dancer and Fu{\v{c}}{\'{\i}}k  (see \cite{DAN,
FUCiK-libro}), for constant weights $m=n$. They were interested in problems with jumping
non-linearities, and they obtained for $m$, $n\equiv 1$ that \eqref{ecu} has a nontrivial
solution corresponding to $(\alpha, \beta)$ if and only if:

\begin{itemize}
\item $\alpha=(\pi/L)^2$, for any $\beta\in\R$,

 \item $\beta=(\pi/L)^2$, for any $\alpha\in\R$,

 \item $(\alpha, \beta)\in\R^2_+$ belong to the following families of hyperbolic-like curves $\mathcal{C}_k$,
\begin{align*}
\mathcal{C}_{k}: & \quad  \frac{k\pi}{2L \sqrt{\alpha}}+\frac{k\pi}{2L \sqrt{\beta}}=1, \qquad k\in \N  \mbox{ even} \\
\mathcal{C}_{k}^+: & \quad \frac{(k+1)\pi}{2L\sqrt{\alpha}}+\frac{(k-1)\pi}{2L\sqrt{\beta}}=1, \qquad k\in \N \mbox{ odd}\\
\mathcal{C}_{k}^-: & \quad \frac{(k-1)\pi}{2L\sqrt{\alpha}}+\frac{(k+1)\pi}{2L\sqrt{\beta}}=1, \qquad k\in \N  \mbox{ odd.}\\
\end{align*}
\end{itemize}

The existence of similar curves in the spectrum was proved later for non-constant weights
$m\not\equiv n$  by Rynne in \cite{RYN} , together with several properties about simplicity of
zeros, monotonicity of the eigenvalues respect to the pair of weights, continuity and
differentiability of the curves, and the asymptotic behavior of the curves when $\alpha$ or
$\beta \to \infty$. For sign-changing weights, similar results were obtained by Alif and
Gossez, and they can be found in \cite{A02, AG01}.

All these results were generalized for half-linear differential equations involving the one
dimensional $p-$Laplacian operator,
\begin{align} \label{ecu-1}
\begin{cases}
   -(|u'|^{p-2}u')'= \alpha m(x) |u^+|^{p-2}u^+ - \beta n(x) |u^-|^{p-2}u^- \qquad x\in (0,L)\\
   u(0)=u(L)=0
\end{cases}
\end{align}
with different boundary conditions, see for instance Section 6.3.3 in \cite{DR-libro} for the
constant coefficient case, or \cite{CCYZ14, AR3} where indefinite weights $m$, $n$ and
different boundary conditions were considered.

By using shooting arguments Rynne proved that the spectrum of problem \eqref{ecu}
can be described as an union of curves
$$\Sigma:= \mathcal{C}_0^\pm \cup \bigcup_{k\in \N } \mathcal{C}_k^\pm  . $$
The curves $\mathcal{C}_0^\pm$ are called the trivial curves since the eigenfunctions does
not changes signs, we can write $\mathcal{C}_0^+=\{\lam_1^m \} \times \R$,
$\mathcal{C}_0^-=\R \times \{\lam_1^n \}$, where $\lam_1^s$ denotes the first eigenvalue
of problem
\begin{equation}\label{ecu2}\begin{cases}
-u''= \lam s(x) u, & \qquad x\in (0,L)\\
u(0)=u(L)=0. & \end{cases}
\end{equation}
The curve $\mathcal{C}_k^+$ (resp. $\mathcal{C}_k^-$) correspond to pairs $(\alpha, \beta)$
with a nontrivial solution having $k$ internal zeros and positive (resp. negative) slope at the
origin. Let us observe that now we have two curves for $k$ even, in the constant coefficient
case both curves coincide but this is not true for general weights.

The curves $\mathcal{C}_k^+$ are not known for general weights $m$, $n$, and only its limit
behavior as $\alpha$ or $\beta\to \infty$ is known. This behavior  is related to the sequence
of eigenvalues $\{\lam_k^s\}_{k\in\N}$ of problem \eqref{ecu2},  and we have that:

\begin{itemize}
\item When $k\in \N$ is even, the curves  $\mathcal{C}_{k}^\pm$ are asymptotic to the line $ \R \times  \lambda_{k/2}^n $ as $\alpha\to \infty$; and to the line $\lambda_{k/2}^m \times  \R $ as $\beta\to \infty$.

\item When $k \in \N$ is odd,  $\mathcal{C}_{k}^+$ is asymptotic to  the line $\R\times \lam_{(k-1)/2}^n$, and  $\mathcal{C}_{k}^-$ is asymptotic to the line  $\R\times \lam_{(k+1)/2}^n$, as $\alpha\to +\infty$.

\item When $k \in \N$ is odd,  $\mathcal{C}_{k}^-$ is asymptotic to  the line $\R\times \lam_{(k+1)/2}^n$, and  $\mathcal{C}_{k}^-$ is asymptotic to the line  $\R\times \lam_{(k-1)/2}^n$, as $\alpha\to +\infty$.
\end{itemize}

Although in some cases the curves $\mathcal{C}_{k}^+$ and $\mathcal{C}_{k}^-$ can
coincide in one or more points, in this work we will think of them as different curves.
That is, we will consider as different two eigenfunctions if one of them starts with
positive slope, and the other one start with negative slope.

\bigskip

Our main objective in this work is to give a description of the asymptotic behavior of the
curves $\mathcal{C}_k^\pm$. Quantitative bounds of the Fu{\v{c}}{\'{\i}}k spectrum are by far less
common in the literature, and the qualitative description states that they are $C^1$ curves.
By using the ideas in \cite{Pi} it is possible to enclose  with an hyperbolic like curve the region
of the plane containing the Fu{\v{c}}{\'{\i}}k eigenpairs whose corresponding eigenfunctions
have $k$ nodal domains. Here, we give an explicit asymptotic description of the curves
together with an error term, in the spirit of Weyl's asymptotic of the Laplacian eigenvalues.

\bigskip
In order to state our results, let us introduce some notations.
We denote by $\mathcal{K}_\theta$ a symmetric region in the first quadrant (with respect to the line $x=y$) between
two rays through the origin forming an angle $\theta\in(0,\pi/2)$. Moreover,  $f(k) \sim g(k)$ as $k\to \infty$ means that $\lim_{k\to \infty}f(k)/g(k)=1$.

In order to improve the asymptotic estimates, we must impose an additional
condition on the weights, related to the Morrey-Campanato space $ L^{\gamma,1}(0,L)$:

\begin{definition}\label{gamacond}
	Given $\gamma> 0$, we say that the function $m$  satisfies the $\gamma$-condition if $m$ is continuous and
there exists two positive constants $m_1$, $m_2$ such that $0< m_1 \le m \le m_2<\infty $
and
 $m\in L^{\gamma,1}(0,L)$, the
 Morrey-Companato space of functions in $L^1(0,L)$ satisfying
	$$
		\int_{I} |m-m_I| \leq C 	|I|^\gamma
	$$ for any $I\subset (0,L)$, where $C$ is a fixed constant and
$$m_I =  \frac{1}{|I|}\int_I m(x)dx$$
is the mean value  of $m$ in $I$.
 \end{definition}

	Observe that, when $m$ is H\"older continuous of order $r >0$ and bounded away from
zero, then it satisfies the $\gamma-$condition for $0\leq \gamma \leq 1+r/2$.

\bigskip

Our first result correspond to the case $m=n$:

\begin{thm}\label{mainmn} Let $\Sigma$ be the Fu{\v{c}}{\'{\i}}k  spectrum of problem \eqref{ecu}, and let $m\equiv n$ be a
continuous function  in $[0,L]$,  bounded below by some some positive constant.
Let $(\alpha_k^\pm,\beta_k^\pm)\in \mathcal{C}_k^\pm$.
 Given some fixed angle $\theta\in(0,\pi/2)$, the curves $\mathcal{C}_k^\pm$ have the following asymptotic behavior in the region
 $\mathcal{K}_\theta$
$$1\sim \frac{ \pi k}{2 }  \left(\frac{(\alpha_k^\pm)^{-\frac{1}{2}} +  (\beta_k^\pm)^{-\frac{1}{2}}}{\int_0^L m^{\frac{1}{2}\, dx}}\right)$$
as $k\to \infty$.
 \end{thm}

For different weights $m$ and $n$ we have:

\begin{thm}\label{main} Let $\Sigma$ be the Fu{\v{c}}{\'{\i}}k  spectrum of problem \eqref{ecu} where $m$, $n$ are
 continuous functions in $[0,L]$,  bounded below by some some positive constant.
Let $(\alpha_k^\pm,\beta_k^\pm)\in \mathcal{C}_k^\pm$. Given some fixed angle $\theta\in(0,\pi/2)$,
 the
curves $\mathcal{C}_k^\pm$ have the following asymptotic behavior  in the region  $\mathcal{K}_\theta$
$$1\sim \frac{ \pi k}{2} \left( \int_0^L \Big((\alpha_k^\pm m)^{-\frac{1}{2}} +  (\beta_k^\pm n)^{-\frac{1}{2}}\Big)^{-1}\, dx \right)^{-1}$$
as $k\to \infty$.
Moreover, if $\big(m^{-\frac{1}{2}}+(n t)^{-\frac{1}{2}} \big)^{-1}$ satisfies the $\gamma$-condition for some $\gamma>1$, then  for
all $\delta \in [\frac{1}{\gamma},1]$, we have
$$1 =  \frac{ \pi k}{2} \left( \int_0^L \Big((\alpha_k^\pm m)^{-\frac{1}{2}} +  (\beta_k^\pm n)^{-\frac{1}{2}}\Big)^{-1}\, dx \right)^{-1} + O(k^{\delta})$$
as $k\to \infty$.

  \end{thm}

\begin{rem}
When $\alpha_k^\pm = \beta_k^\pm$, and $m=n$,  the pair of curves intersect at an eigenvalue $\lam_k$ of the
Sturm-Liouville problem \eqref{ecu2} with $s\equiv m$,  and we recover the well known
asymptotic formula of Weyl,
$$
\lam_k \sim  \left( \frac{ \pi k }{   \int_0^L m^{\frac{1}{2}}\, dx }\right)^2.
$$
\end{rem}

\bigskip

We divide the proof of Theorem   \ref{main}
 in several steps.  The first step is to fix a line $(\lam, t\lam)$ in the first quadrant
and to estimate the number of intersections with the curves in $\Sigma$ as a function of $\lambda$. To this end, we need to study
the eigenvalue  problem
\begin{equation}\label{ecux}
\begin{cases}
   - u''  =\lam (  m(x) u^+ - t  n(x) u^-  )& \qquad x\in [0,L] \\
   u(0)  =u(L)=0. &
\end{cases}
\end{equation}
These eigenvalues are related to the half-eigenvalues studied for instance in \cite{CCYZ14,
GR11}. Observe that only a positive multiple of a solution is itself a solution corresponding to
the same eigenvalue, if a negative multiple is a solution too, we will count this as a double
eigenvalue.

From the description of the  Fu{\v{c}}{\'{\i}}k  spectrum of problem \eqref{ecu}, the spectrum
of problem  \eqref{ecux} consists in a double sequence
\begin{equation} \label{autov_d}
0<\lam_{1,t}^\pm < \lam_{2,t}^\pm < \lam_{3,t}^\pm < \cdots < \lam_{2k-1,t}^\pm < \lam_{2k,t}^\pm \cdots \nearrow \infty
\end{equation}
for each fixed $t>0$. Each eigenvalue $\lam_{2,t}^\pm$ has a unique associated eigenfunction, normalized by
$\pm u'(0)=1$. The eigenfunction corresponding to $\lam_{k,t}^\pm$ has precisely $k$ nodal
domains on $(0,L)$, and $k+1$ zeros in $[0, L]$. Moreover, for $t=1$ we have
$\lam^+_{k,1}=\lam_{k,1}^-$.

We introduce the spectral counting function $N(\lam, t, (0,L))$, defined as
$$
N(\lam, t, (0,L)) = \# \{ k : \lam_{k,t}^+(0,L)\le \lam\} + \# \{ k : \lam_{k,t}^-(0,L)\le \lam\}.
$$
We emphasize  the dependence of the eigenvalues  on the interval, since in order to count
these eigenvalues, we use a bracketing argument comparing them with the eigenvalues in
subintervals. However, since the problem has no variational structure, we cannot use the
Dirichlet-Neumann bracketing of Courant \cite{CH}. Moreover, since the Sturmian oscillation
theory does not hold for Fu{\v{c}}{\'{\i}}k  eigenvalues, and it is not true that between two
zeros of an eigenfunction there is at least a zero of any eigenfunction corresponding to a
higher eigenvalue, we cannot use a different bracketing based on the number of zeros nor
the Sturm-Liouville theory that can be found in \cite{Pi1}, and we need to prove the following
theorem:

\begin{thm} \label{prop_suma}
Let $N(\lam, t, (0,L))$ be the spectral counting function of problem \eqref{ecux}  where $m$, $n$ are
 continuous functions in $[0,L]$,  bounded below by some some positive constant.
Let $c\in(0,L)$. Then for any fixed $t>0$,
$$N(\lam,t, (0,L))= N(\lam,t,(0,c)) +N(\lam,t,(c,L))+ O(1)$$
when $\lam\to \infty$.
\end{thm}

From Theorem \ref{prop_suma} we derive the following Weyl type result:

\begin{thm}\label{ndelam} Let $N(\lam, t, (0,L))$ be the spectral counting function on $(0,L)$ of problem
\eqref{ecux}, where $m$, $n$ are
 continuous functions  in $[0,L]$,  bounded below by some some positive constant. For any fixed $t>0$,
	$$
		N(\lam, t,(0,L))=\frac{4\sqrt{\lam}}{\pi} \int_0^L  \big(m^{-\frac{1}{2}}+(n t)^{-\frac{1}{2}} \big)^{-1} \, dx + o(\sqrt{\lam})
	$$
as $\lam\to \infty$.
\end{thm}

\bigskip

Finally, let us show that the error term in Theorem \ref{ndelam} can be improved for more regular weights satisfying the $\gamma$-condition
introduced in Definition \ref{gamacond}, following the ideas  in \cite{F-L}.

\newpage

\begin{thm}\label{campanato}  Suppose that $\big(m^{-\frac{1}{2}}+(n t)^{-\frac{1}{2}} \big)^{-1}$ satisfies the $\gamma$-condition for some $\gamma>1$.
	Then for each $t>0$ fixed and  for all $\delta \in [\frac{1}{\gamma},1]$, we have
	$$
		N(\lam,(0,L))=\frac{4\sqrt{\lam}}{\pi} \int_0^L  \big(m^{-\frac{1}{2}}+(n t)^{-\frac{1}{2}} \big)^{-1} \, dx + O(\lam^{\delta/2}).
	$$
\end{thm}

Let us observe that the arguments here can be easily modified and applied to problem
\eqref{ecu-1}, and also for other homogeneous boundary conditions. However, those
arguments are strongly dependent on the one dimensional nature of the problem, and we
cannot expect an easy generalization to higher dimensional problems where, on the other
hand, several strange phenomena can occur.

\bigskip

The paper is organized as follows. In Section 2 we review some necessary facts about the
eigenvalues and we prove the bracketing argument in  Theorem \ref{prop_suma}. We prove Theorems \ref{main}, \ref{ndelam} and \ref{campanato}
in Section 3. In Section 4  we present some numerical experiments.

\section{A fixed line}

Let us study the distribution of the intersection between $\Sigma$ and a line $\beta=t
\alpha$ through the origin for a fixed $t>0$. Let us call $\alpha=\lam$ and $\beta=t\alpha=t
\lam$, and let us describe the spectrum of problem \eqref{ecux}.

We consider first the case of constant weights $m$ and $n$. Hence, the intersection $\Sigma\cap \{( \lam, t\lam) : \lam >0\}$ can be computed explicitly:

\begin{lema}
Let $\{\lam_k^\pm\}_{k\in\N}$ be the eigenvalues of \eqref{ecux} with constant weights in
$(0,L)$. Then, for each fixed $t>0$, we have
\begin{align} \label{autovcte}
\begin{split}
&(\lam_{k,t}^\pm)^\frac{1}{2}=\frac{k\pi}{2L} \left(m^{-\frac{1}{2}}+(nt)^{-\frac{1}{2}} \right), \qquad \qquad \qquad \qquad k \mbox{ even}\\
&(\lam_{k,t}^+)^\frac{1}{2}=\frac{\pi}{2L} \left((k+1)m^{-\frac{1}{2}}+(k-1)(nt)^{-\frac{1}{2}} \right), \qquad k \mbox{ odd}\\
&(\lam_{k,t}^-)^\frac{1}{2}=\frac{\pi}{2L} \left((k-1)m^{-\frac{1}{2}}+(k+1)(nt)^{-\frac{1}{2}} \right), \qquad k \mbox{ odd.}
\end{split}
\end{align}
\end{lema}

We omit the proof since it follows from the explicit formula for the Fu{\v{c}}{\'{\i}}k
eigenvalues with constant weights given in the Introduction.

\bigskip
We are interested now in the spectral counting function $N(\lam, t, (0,L))$ for any fixed
$t>0$. In order to prove Theorem \ref{prop_suma}, we review the characterization of the
eigenvalues in terms of the Pr\"ufer angle.

We introduce two auxiliary functions $\rho$, $\varphi$ and we propose the following Pr\"ufer
type transformation
\begin{align} \label{aaaa}
  \begin{cases}
    u(x) &=\sqrt{\lam m(x)} \sin( \varphi(x))^+-  \sqrt{\lam t n(x)}\sin( \varphi(x))^-,\\
    u'(x)&= \rho(x) \cos( \varphi(x)),
  \end{cases}
\end{align}
and an straightforward computation shows that $\rho(x)$ and $\varphi(x)$ satisfy the system
of ordinary differential equations
\begin{align}\label{fi}
\varphi'(x)=
\begin{cases}
\sqrt{\lam m(x)} + \frac{1}{2}\frac{m'(x)}{m(x)}\cos(\varphi(x)) \sin(\varphi(x)) &\qquad \mbox{if } \sin(\varphi(x))\geq 0\\
\sqrt{\lam n(x)} + \frac{1}{2}\frac{n'(x)}{n(x)}\cos(\varphi(x)) \sin(\varphi(x))  &\qquad \mbox{if } \sin(\varphi(x))< 0,
\end{cases}
\end{align}

\begin{align}
\rho'(x)=
\begin{cases}
\frac{1}{2}\frac{m'(x)}{m(x)}\rho(x) \sin^2(\varphi(x)) &\qquad \mbox{if } \sin(\varphi(x))\geq 0\\
\frac{1}{2}\frac{n'(x)}{n(x)}\rho(x) \sin^2(\varphi(x)) &\qquad \mbox{if } \sin(\varphi(x))<0.
\end{cases}
\end{align}

Now, a shooting argument enable us to compute the eigenvalues since the corresponding
eigenfunctions have a given number of zeros. The  zeros of $u$ depend on the zeros of
$\sin(\varphi(x))$ since the function $\rho$ is strictly positive. Whenever $\varphi(x)$ is an
integer multiple of $\pi$, we have $\varphi'(x)>0$. Hence, we have the following
characterization of the eigenvalues:
$$\lam_k^+=\min\{\lam\,:\, \varphi(\lam, 0) =0, \varphi(\lam,L)=k\pi \}, \qquad \quad$$
$$\lam_{k}^-=\min\{\lam\,:\, \varphi(\lam, 0) =\pi,  \varphi(\lam,L)=(k+1)\pi  \}.$$

\bigskip

The next Lemma is a version of the well-known Sturm's Comparison Theorem for the
half-eigenvalues of \eqref{ecux}, and it is an easy consequence of the previous
characterization of eigenvalues in terms of the Pr\"ufer transformation, see Theorem 5.3 in
\cite{RYN}.

\begin{lema}\label{lemapeso} Let $m_i$, $n_i$ $i=1,2$ positive and continuous weights such that $m_1\geq m_2$ and $n_1\geq n_2$, and let us call $\lam_{k,t}^\pm(m_i,n_i)$ the corresponding half-eigenvalues of problem \eqref{ecux}. Then $\lam_{k,t}(m_1,n_1)\leq \lam_{k,t}(m_2,n_2)$.
\end{lema}

As a consequence, we have the following bound for the Fu{\v{c}}{\'{\i}}k  eigenvalues with
two weights in terms of the eigenvalues of a weighted problem:

\begin{lema}  Let $m$, $n$ be positive and continuous weights and let us call $\lam_{k,t}^\pm(m,n)$ the corresponding half-eigenvalues of problem \eqref{ecux}. Then $\lam_{k,t}(m,n)\leq \mu_{k}$, where $\mu_k$ is the
$k$-th eigenvalue of the problem:
 $$
 -v'' = \mu \frac{mnt}{m+nt}v,
 $$
with zero Dirichlet boundary conditions.
\end{lema}

\begin{proof}
We have the following inequalities,
$$
\frac{mnt}{m+nt}\le\frac{mnt}{nt} = m
$$
$$
\frac{mnt}{m+nt}\le\frac{mnt}{m} = nt
$$
and by Lemma \ref{lemapeso}, we have
$$
\lam_{k,t}^\pm(m, n) \le \lam_{k,t}^\pm\left(\frac{mnt}{m+nt},\frac{mnt}{m+nt}\right),
$$
which coincides with $\mu_k$.
\end{proof}

\bigskip
A critical step in the proof of Theorem \ref{prop_suma} is the following Lemma, which is a
consequence of the Sturmian oscillation theory in classical eigenvalue problems:

\begin{lema}\label{zeros}
 Let $\lam>\lam_{k,t}^+$  (resp.,  $\lam>\lam_{k,t}^-$), and let $\varphi(\lam, x)$ corresponding to a solution of
equation \eqref{ecux} satisfying $u(0)=0$ and $u(0^+)>0$ (resp., $u(0^+)<0$). Then
$\varphi(\lam,L) > k\pi$ (resp., $\varphi(\lam,L) > (k+1)\pi$).
\end{lema}

\begin{proof}
Let us consider only the case $u(0^+)>0$, the other one is similar. Let us denote by $\{0,x_1,
\cdots, x_{k-1}, L\}$ the zeros of an eigenfunction corresponding to $\lam_{k,t}^+$.

In the first nodal domain $(0,x_1)$ we can apply the classical Sturmian theory, so the
solution $u$ has a first zero at $0$, and another zero $y_1< x_1$.

Now, the next zero $y_2$ must be located before $x_2$, or the second nodal domain starting
in $y_1$ strictly contains the interval $(x_1,x_2)$. However, this is not possible, since
$\lam_k^+<\lam$ and the Sturm-Liouville theory implies that $u$ must have a zero between
$x_1$ and $x_2$.

In much the same way, the nodal domain starting at $y_2<x_2$ cannot contain the full
interval $(x_2, x_3)$, and applying this argument repeatedly, we obtain that $y_k<L$.

Finally, since $\varphi$ cannot decrease at a multiple of $\pi$, we have $\varphi(\lam, L)>k
\pi$ and the Lemma is proved.
\end{proof}

\begin{proof}[Proof of Theorem \ref{prop_suma}]

Let us consider the following auxiliary eigenvalue problems in $(0,c)$ and $(c,L)$, with the
original boundary conditions in $0$ and $L$, and Neumann boundary conditions at $c$:

\begin{align} \label{pp1}
\begin{split}
\begin{cases}
&- u'' =\mu  (m(x) u^+ - t n(x) u^- ) \qquad
  x\in (0,c)\\
&u(0)=u'(c)=0,
\end{cases}
\end{split}
\end{align}

\begin{align} \label{pp2}
\begin{split}
\begin{cases}
&- u'' =\mu  (m(x) u^+ - t n(x) u^- ) \qquad
  x\in  (c,L)\\
&u'(c)= u(L)=0.
\end{cases}
\end{split}
\end{align}

We have two double sequences of eigenvalues $\{\mu_{k,t}^{\pm}(0,c)\}_{k\ge1}$,
$\{\mu_{k,t}^{\pm}(c,L)\}_{k\ge1}$ corresponding to problems \eqref{pp1} and \eqref{pp2}.

Let us fix  $\lam$, and there exists some integer $n$ such that
 $$\lam_{n-1,t}^-<\lam_{n,t}^+ \le
\lam <\lam_{n+1,t}^+<\lam_{n+2,t}^-.$$
Moreover, let us note that there exist an even integer $j$ such that $\mu_{j,t}^+(0,c) \le\lam < \mu_{j+2,t}^+(0,c)$, and an integer
 $h$ such that $\mu_{h,t}^+(c,L)\le \lam <\mu_{h+1,t}^+(c,L)$.

Observe that $\lam_{n,t}^+ \le \lam$ does not implies $\lam_{n,t}^-\le\lam$, and by using similar arguments in all the cases we obtain
the following bounds:
\begin{align} \label{enes}
\begin{split}
2n-1 &\le N(\lam, t, (0,L)) \le 2n+1 \\
2j-1 & \le N(\lam, t, (0,c)) \le 2(j+1)+1 \\
2h-1 & \le N(\lam, t, (c,L)) \le 2h+1.
\end{split}
\end{align}

We claim now that
$$-9\le N(\lam,t, (0,L))) - N(\lam,t,(0,c))-N(\lam,t,(c,L)) \le  11.$$

The eigenfunctions corresponding to $\mu_{j,t}^+(0,c) $ and $\mu_{h,t}^+(c,L) $  have $j+1$ zeros
in $[0,c]$ and $h+1$ in $[c, L]$ respectively. Since both solutions have a zero at $c$, we have
$j+h+1$ zeros in $[0, L]$.

Now, let $\mu = \max\{ \mu_{j,t}^+(0,c), \mu_{h,t}^+(c,L) \}$, and let us solve equation \eqref{fi}
for $\lambda=\mu$. We start the shooting at $0$ if $\mu =   \mu_{j,t}^+(0,c)$ with
$\varphi(\mu,0)=0$, or we solve it backwards starting from $L$ if $\mu = \mu_{h,t}^+(c,L)$ with
$\varphi(\mu,L)=\pi$.  Hence, Lemma \ref{zeros} implies that this solution has at least
$i+j+1$ zeros in $[0, L]$.

We compare now this solution with the one corresponding to $\lam_{n+1,t}^+$ (if  $\mu =
\mu_{j,t}^+(0,c)$) or the one corresponding to $\lam_{n+2,t}^-$ (if  $\mu = \mu_{h,t}^+(c,L)$). In
both cases, we get
\begin{equation}\label{ineq}
j+h+1 \le n+3.
\end{equation}

We can obtain another inequality starting with the eigenfunctions corresponding to
$\mu_{j+2,t}^+(0,c) $ and $\mu_{h+1,t}^+(c,L)$, they have $j+3$ zeros in $[0,c]$ and $h+2$ in
$[c, L]$,  respectively. Hence, we have $j+h+4$ zeros in $[0, L]$. Comparing as before, by
choosing now $\mu = \min\{\mu_{j+2,t}^+(0,c) , \mu_{h+1,t}^+(c,L)\}$, the corresponding
solution has lower than $j+h+4$ zeros, and the zeros of the eigenfunctions corresponding to
$\lam_{n,t}^+$ or $\lam_{n-1,t}^-$ can be bounded below by $n$. So, we get
\begin{equation}\label{ineq2}
n\le j+h+4.
\end{equation}

Therefore, inequalities \eqref{ineq} and \eqref{ineq2}, together with inequalities \eqref{enes}  implies the claim,
and the proof is finished.
\end{proof}

\section{The proof of the main theorems}

In this section we prove Theorems \ref{main}, \ref{ndelam} and \ref{campanato}.
 For convenience, we prove first lower and upper bounds for the eigenvalue counting function, and we introduce the following notation:  given two functions $m$, $n$ we denote
$$d_t(m,n)=m^{-\frac{1}{2}}+ (n t)^{-\frac{1}{2}}.$$

\begin{lema}\label{lemasuarez}
Let $t>0$ be fixed, and let $\{\lam_{k,t}^\pm\}_{k\in\N}$ be the eigenvalues of problem \eqref{ecux}
in $(0,L)$ and suppose that $\underline{m}\leq m \leq \overline{m}$, $\underline{n}\leq n\leq
\overline{n}$. Then
\begin{align*}
 &N(\lam,t,(0,L)) \geq \frac{4\sqrt{\lam} L}{\pi}    \big(d_t(\underline{m},\underline{n})\big)^{-1} -1,  \\
 &  N(\lam,t,(0,L))  \leq  \frac{4\sqrt{\lam}  L}{\pi}  \big(d_t(\overline{m},\overline{n})\big)^{-1} + 2.
\end{align*}
\end{lema}

\begin{proof}The proof follows as a consequence of the monotonicity of eigenvalues respect to the weight in Lemma \ref{lemapeso}.
For $k$ even we have that
\begin{equation}
\frac{k^2\pi^2}{4L^2}\big(d_t(\overline{m},\overline{n}) \big)^2    \leq \lam_{k,t}^\pm \leq \frac{k^2\pi^2}{4L^2} \big(d_t(\underline{m},\underline{n}\big) \big)^2.
\end{equation}
Consequently, if we consider only the even curves,
\begin{align*}
\#\{ \lam_{k,t}^\pm \leq \lam, k \textrm{ even} \} &\geq   \#\Big\{k:  \frac{k^2\pi^2}{4L^2}  d_t(\underline{m},\underline{n})^2 \leq \lam , k \mbox{ even} \Big\}  \\
\#\{ \lam_{k,t}^\pm \leq \lam, k \textrm{ even} \} &\leq  \#\Big\{k:  \frac{k^2\pi^2}{4L^2} d_t(\overline{m},\overline{n})^2 \leq \lam  , k \mbox{ even} \Big\}.
\end{align*}

Hence, denoting by $\lfloor{x}\rfloor$
the  largest integer not greater than  $x$, we have
\begin{align*}
\#\{ \lam_{k,t}^\pm \leq   \lam, k \textrm{ even} \} &\geq
\left\lfloor
\frac{ 2\sqrt{\lam} L }{\pi}    d_t(\underline{m},\underline{n})^{-1} \right\rfloor \\
\#\{ \lam_{k,t}^\pm \leq \lam, k \textrm{ even} \} &\leq \left\lfloor \frac{ 2 \sqrt{\lam} L }{\pi}  d_t(\overline{m},\overline{n})^{-1}\right\rfloor.
\end{align*}

The lower bound follows by bounding  $\lfloor{x}\rfloor \ge x-1$, and the other one follows
since $\lfloor{x}\rfloor\le  x$ and there is a pair of odd curves between
two pairs of even curves. The Lemma is proved.
\end{proof}

We are ready to prove Theorem \ref{ndelam}. Essentially, we split $(0,L)$ in finitely many
intervals, and then we apply Theorem \ref{prop_suma} and Lemma \ref{lemasuarez}.

\begin{proof}[Proof of Theorem \ref{ndelam}]

Let us split $(0,L)$ as
$$(0,L)=\cup_{1\leq j \leq J} I_j,$$
 where $I_j \cap I_k=\emptyset$, and $|I_j|=L/J = \eta$. We define
$$\underline{m}_j=\inf_{x\in I_j} m(x), \quad \overline{m}_j=\sup_{x\in I_j} m(x).$$
Since $m$ and $n$ are continuous functions they are Riemann integrable, and given any $\ve>0$, we can choose $\eta >0$ such that
\begin{align*}
&\sum_{j=1}^J \eta\big((\underline{m}_j)^{-\frac{1}{2}}+(\underline{n}_j t)^{-\frac{1}{2}} \big)^{-1} \ge \int_0^L \big(m^{-\frac{1}{2}}+(n t)^{-\frac{1}{2}} \big)^{-1} \,dx  - \ve, \\
&\sum_{j=1}^J \eta\big((\overline{m}_j)^{-\frac{1}{2}}+(\overline{n}_j t)^{-\frac{1}{2}} \big)^{-1} \le \int_0^L \big(m^{-\frac{1}{2}}+(n t)^{-\frac{1}{2}} \big)^{-1} \,dx  + \ve.
\end{align*}

From Theorem \ref{prop_suma} we have that
$$N(\lam,t,(0,L)) =  \sum_{j=1}^J N(\lam, t, I_j) + O(J),$$
and Lemma \ref{lemasuarez} implies
\begin{align*}
 \sum_{j=1}^J N(\lam, t, I_j) &\leq  \sum_{j=1}^J  \frac{4\eta\sqrt{\lam}  }{\pi}  \big((\overline{m}_j)^{-\frac{1}{2}}+(\overline{n}_j t)^{-\frac{1}{2}} \big)^{-1} +O(J)\\
&= \frac{4\sqrt{\lam}  }{\pi}  \Big(\int_0^L \big(m^{-\frac{1}{2}}+(n t)^{-\frac{1}{2}} \big)^{-1} \,dx + \ve\Big) + O(J)
\end{align*}

\begin{align*}
 \sum_{j=1}^J N(\lam, t, I_j) &\geq  \sum_{j=1}^J  \frac{4\eta\sqrt{\lam}  }{\pi}  \big((\underline{m}_j)^{-\frac{1}{2}}+(\underline{n}_j t)^{-\frac{1}{2}} \big)^{-1}-O(J) \\
&= \frac{4\sqrt{\lam}  }{\pi}  \Big(\int_0^L \big(m^{-\frac{1}{2}}+(n t)^{-\frac{1}{2}} \big)^{-1} \,dx  - \ve\Big) -O(J)
\end{align*}
and the proof is complete.
\end{proof}

Let us prove now Theorem \ref{campanato}.

\begin{proof}[Proof of Theorem \ref{campanato}]
 Let $0<\eta<\eta_1<0$ be fixed, and let us split  $(0,L)=\cup_{1\leq j \leq J} I_j$  as before,
 where $I_j \cap I_k=\emptyset$, $|I_j|=L/J = \eta$.

 Let us  define, for $1\le i \le J$
$$
	\varphi(\lam)=\frac{2\sqrt{\lam}}{\pi}  \int_0^L d_t(m,n), \qquad
	\varphi(\lam,I_j)= \frac{2\eta  \sqrt{\lam} d_{I_j}}{\pi}
$$
where $d_{I_j}=|I_j|^{-1} \int_{I_j} d_t(m,n)$.

From Theorem \ref{prop_suma} we obtain
\begin{equation} \label{dd1}
	\sum_{ j=1}^J N(\lam,t,I_j,d_t) -\varphi(\lam) - O(J) \leq
	N(\lam, t,(0,L),d_t) -\varphi(\lam)
\end{equation}
\begin{equation} \label{dd2}
	N(\lam,t,(0,L),d_t) -\varphi(\lam)  \leq
	\sum_{ j=1}^J  N(\lam,t,I_j,d_t) -\varphi(\lam) +  O(J).
\end{equation}

Let us bound the left hand side of inequality \eqref{dd1}; we have
\begin{align*}
\begin{split}
	\sum_{j=1}^J N(\lam,t,I_j,d_t) -\varphi(\lam) &\leq
	\sum_{j=1}^J N(\lam,t,I_j,d_{I_j})-\varphi(\lam,I_j)\\ &\quad +\sum_{j=1}^J \varphi(\lam,I_j) - \varphi(\lam) \\ & \quad +
	\sum_{j=1}^J N(\lam,t,I_j,d_t)-\sum_{j=1}^J N(\lam,t,I_j,d_{I_j}).
\end{split}	
\end{align*}

The first term can be bounded by using Lemma \ref{lemasuarez},
$$
	\sum_{j=1}^J \Big|	 N(\lam,t,I_j,d_{I_j})-\varphi(\lam,I_j)\Big| \leq J M L,
$$
where $M$ is the maximum of $d_t$ in $(0,L)$.

Rewriting the second term, we can bound it as
$$
	\sum_{j=1}^J \varphi(\lam,I_j) - \varphi(\lam) = \pi^{-1}2\sqrt{\lam} \sum_{j=1}^J \int_{I_j} \Big( d_t- d_{I_j}\Big).
$$
The assumed regularity of the weight and the $\gamma-$condition gives
$$ \sum_{j=1}^J \varphi(\lam,I_j) - \varphi(\lam) = \pi^{-1}\sqrt{\lam} J c_1 \eta^\gamma.
$$

The third term can be handled using the monotonicity of the eigenvalues respect to the weight and the additivity of the eigenvalue counting function, Lemma \ref{lemasuarez} and Theorem \ref{prop_suma}, since $d_t\leq d_{I_j}+ |d-d_{I_j}|$,
$$
	N(\lam,t,I_j,d_t) \leq N(\lam,t,I_j, d_{I_j}) +N(\lam,t,I_j ,|d_t-d_{I_j}|) + O(J),
$$
which gives
\begin{align*}
\begin{split}
	\sum_{j=1}^J N(\lam,t,I_j,d_t)-N(\lam,t,I_j,d_{I_j}) & \leq \sum_{j=1}^J N(\lam,t,I_j,|d_t-d_{I_j}|) \\
& \leq \sum_{j=1}^J
\frac{2\sqrt{\lam}}{\pi}\int_{I_j}|d_t-d_{I_j}| +O(J) \\
& \le
\pi^{-1} 2\sqrt{\lam} J \eta^\gamma +  O(J),
\end{split}
\end{align*}
 we have used the same arguments as above, the regularity of the weight and the $\gamma-$condition.

Collecting terms, and by using that $J=L \eta^{-1}$, we can bound the left hand side of
\eqref{dd1},
\begin{equation} \label{exx1}
	\left| \sum_{ j=1}^J N(\lam,t,I_j,d_t) -\varphi(\lam) - O(J) \right| \le  C (\sqrt{\lam}\eta^{\gamma-1}+\eta^{-1}).
\end{equation}

In much the same way, we can bound the right-hand side of inequality \eqref{dd2}, since we only need to change the constant in the $O(J)$ term, obtaining
\begin{equation} \label{exx2}
\sum_{ j=1}^J  N(\lam,t,I_j,d_t) -\varphi(\lam) +  O(J) \le C (\sqrt{\lam}\eta^{\gamma-1}+\eta^{-1}).
\end{equation}
Hence, we get from \eqref{exx1} and \eqref{exx2}
$$
	\left|N(\lam,t,(0,L),d_t) -\varphi(\lam)\right|\leq C (\sqrt{\lam}\eta^{\gamma-1}+\eta^{-1}).
$$

We choose now $\eta=  \lam^{-\alpha/2}$ with $0<\alpha<1$. We can have $\sqrt{\lam}\eta^{\gamma-1} \sim \sqrt{\lam}^\delta$ and $\sqrt{\lambda}^\alpha \sim \sqrt{\lam}^\delta$  only for $\delta\in [\frac{1}{\gamma},1]$,  and the proof is finished.
\end{proof}

By taking $\lam\sim \lam_{k,t}^\pm$ in Theorem \ref{ndelam}, the following asymptotic behavior of the
eigenvalues is obtained:

\begin{cor}
Given a fixed $t\in\R$, the following asymptotic behavior holds:
\begin{equation} \label{asintt}
(\lam_{k,t}^\pm)^\frac{1}{2} \sim \frac{ \pi k}{2} \left( \int_0^L \big(m^{-\frac{1}{2}} +  (tn)^{-\frac{1}{2}}\big)^{-1}\, dx \right)^{-1}.
\end{equation}
\end{cor}

It follows by observing that $N(\lam_{k,t}^\pm) \sim 2 k$.

\begin{rem}\label{asintotico}
 We have obtained a generalization of \eqref{autovcte}. When $k$ is even, we recover the same formula for constant coefficients. However, for $k$ odd, there is an error term which is of order $O(1)$.

Moreover, when $m=n$, we have
$$\lam_{k,t}^\frac{1}{2} \sim \frac{ k\pi}{2} \big(1+t^{-\frac{1}{2}}\big) \left(\int_\Omega m^\frac{1}{2} \, dx \right)^{-1} \sim
\lam_{k}^\frac{1}{2}\big(1+t^{-\frac{1}{2}}\big),$$ where $\lam_k$ is the $k-$th
eigenvalue of problem \eqref{ecu2}, recovering the classical Weyl's asymptotic expression
for the eigenvalues of the weighted problem.
\end{rem}

We can prove now the asymptotic expression of the curves in the Fu{\v{c}}{\'{\i}}k  spectrum.

 \begin{proof}[Proof of Theorem \ref{main}]
Recall that $\mathcal{K}_\theta$ is a symmetric region in the first quadrant between
two rays through the origin forming an angle $\theta\in(0,\pi/2)$ . We can rewrite Eq.
\eqref{asintt} as
$$1 \sim \frac{ \pi k}{2} \left( \int_0^L \left((\lam_{k,t}^\pm  m)^{-\frac{1}{2}} +
(\lam_{k,t}^\pm  tn)^{-\frac{1}{2}}\right)^{-1}\, dx \right)^{-1},
$$
and by calling $\alpha_k^\pm=\lam_{k,t}^\pm$, $\beta_k^\pm=\ t\lam_{k,t}^\pm$, we obtain the desired
asymptotic behavior of the curves $\mathcal{C}_k^\pm$  inside $\mathcal{K}_\theta$:
$$1\sim \frac{ \pi k}{2} \left( \int_\Omega \big((\alpha_k^\pm m)^{-\frac{1}{2}} +  (\beta_k^\pm n)^{-\frac{1}{2}}\big)^{-1}\, dx \right)^{-1}$$
as $k\to\infty$.

In order to improve the remainder estimate, let us observe that, for each $t$ fixed,
Eq. \eqref{asintt} implies that there exists a constant $C(t)$ such that
$$\lam_{k,t}^\pm  \sim C(t) k^2.$$
Now,
 Theorem \ref{campanato}  gives
$$
		2k = N(\lambda_k^\pm,(0,L))=\frac{4}{\pi} \int_0^L  \big((m \lambda_k^\pm )^{-\frac{1}{2}}+(n \lambda_k^\pm t)^{-\frac{1}{2}} \big)^{-1} \, dx + O(k^\delta),
$$
which is equivalent to
$$
		1 =  \frac{\pi k}{2} \int_0^L  \big((m\alpha_k^\pm)^{-\frac{1}{2}}+(n \beta_k^\pm )^{-\frac{1}{2}} \big)^{-1} \, dx + O(k^\delta),
	$$
since $\alpha_k^\pm=\lambda_k^\pm$, $\beta_k^\pm=t\lambda_k^\pm$.   However, since the constants in the error term depend on
$t$, we
 can obtain uniform bounds only
for $t$ in a compact interval bounded away from zero and infinity, that is, the result holds
only on cones $\mathcal{K}_\theta$ not touching the edges.

The proof is finished.
\end{proof}

\section{Some numeric computations}

It is possible to compute the weighted Fu{\v{c}}{\'{\i}}k  eigenvalues numerically as
in \cite{BrRe}, where Brown and Reichel proposed an algorithm based on Newton's method,
by using the polar coordinates $r$, $\varphi$ as in equation \eqref{aaaa}, solving
 the ordinary differential Eq. \eqref{fi}.

Here, we have used a slightly different idea. The following pseudo code can be easily
implemented, and given two bounds $\lambda_*$, $\lambda^*$, which can be obtained
explicitly from the formulas for the constant coefficient problem and Lemma \ref{lemapeso},
a bisection argument computes the eigenvalue in a fixed line of slope $t$ with the desired
accuracy $\ve$. We start with $\lambda=(\lambda^*-\lambda_*)/2$, and we  solve
 $$
\varphi'(x)=
\sqrt{\lam f(x)} + \frac{1}{2}\frac{f'(x)}{f(x)}\cos(\varphi(x)) \sin(\varphi(x))
$$
alternating between the weights $f=m$ and $f=n$ whenever a multiple of $\pi$ is reached. If
$\varphi=k\pi$ is obtained, we set $\lambda_*=\lambda$, if we reach the extreme $L$  of
the interval and $\varphi(L)<k\pi$, we set $\lambda^*=\lambda$; we restart the process until
the difference $\lam^*-\lam_*$ is less than a prefixed tolerance error. We omit the error
estimate which came from the numerical solution of the ordinary differential equation, which
can be handled as in \cite{BrRe}.

\begin{align} \label{algo}
\begin{split}
  &\textbf{inputs: } k,m,n,t,\ve, \lam^*, \lam_* \\
  &\textbf{**********************}\\
  &k\texttt{: number of eigenvalue, } m,n \texttt{: weights, } t \texttt{: slope of the line, } \\
  &\ve \texttt{: accuracy, } \lam^*, \lam_* \texttt{: bounds of } \lam_k \\
  &\textbf{**********************}\\
  &\texttt{while } \lam^*-\lam_*>\ve\\
  &\qquad\lam=(\lam^*+\lam_*)/2\\
  &\qquad lastzero=0\\
  &\qquad \texttt{for } i=1 \texttt{ to }k\\
  &\qquad \qquad \texttt{if } i \textrm{ mod } 2\neq 0 \texttt{ then } f(x)=m(x) \texttt{ else } f(x)=tn(x)\\
  &\qquad \qquad \texttt{find } \varphi(x) \texttt{ by solving } \eqref{fi} \texttt{ in	} (lastzero,1) \\
  &\qquad \qquad \texttt{find } w \texttt{ such that } \varphi(w)=\pi\\
  &\qquad \qquad lastzero=w\\
  &\qquad \texttt{end}  \\
  &\qquad \texttt{if  } lastzero<1 \texttt{ then }\lam^*=(\lam^*+\lam_*)/2 \texttt{ else } \lam_*=(\lam^*+\lam_*)/2\\
  &\texttt{end}  \\
  &\textbf{output: } \lam \\
\end{split}
\end{align}

Now, in order to compare the asymptotic results with the computed values of the
eigenvalues,  we consider problem \eqref{ecux} in $[0,1]$, with the weights
$m(x)=1+(x+1)^{-1}$ and $n(x)=1+cos^2(2x)$. In Figure 1 we show the first curves
of the Fu{\v{c}}{\'{\i}}k spectrum by using the algorithm \eqref{algo}.

\begin{figure}[ht]\label{figura}
\includegraphics[width=8cm]{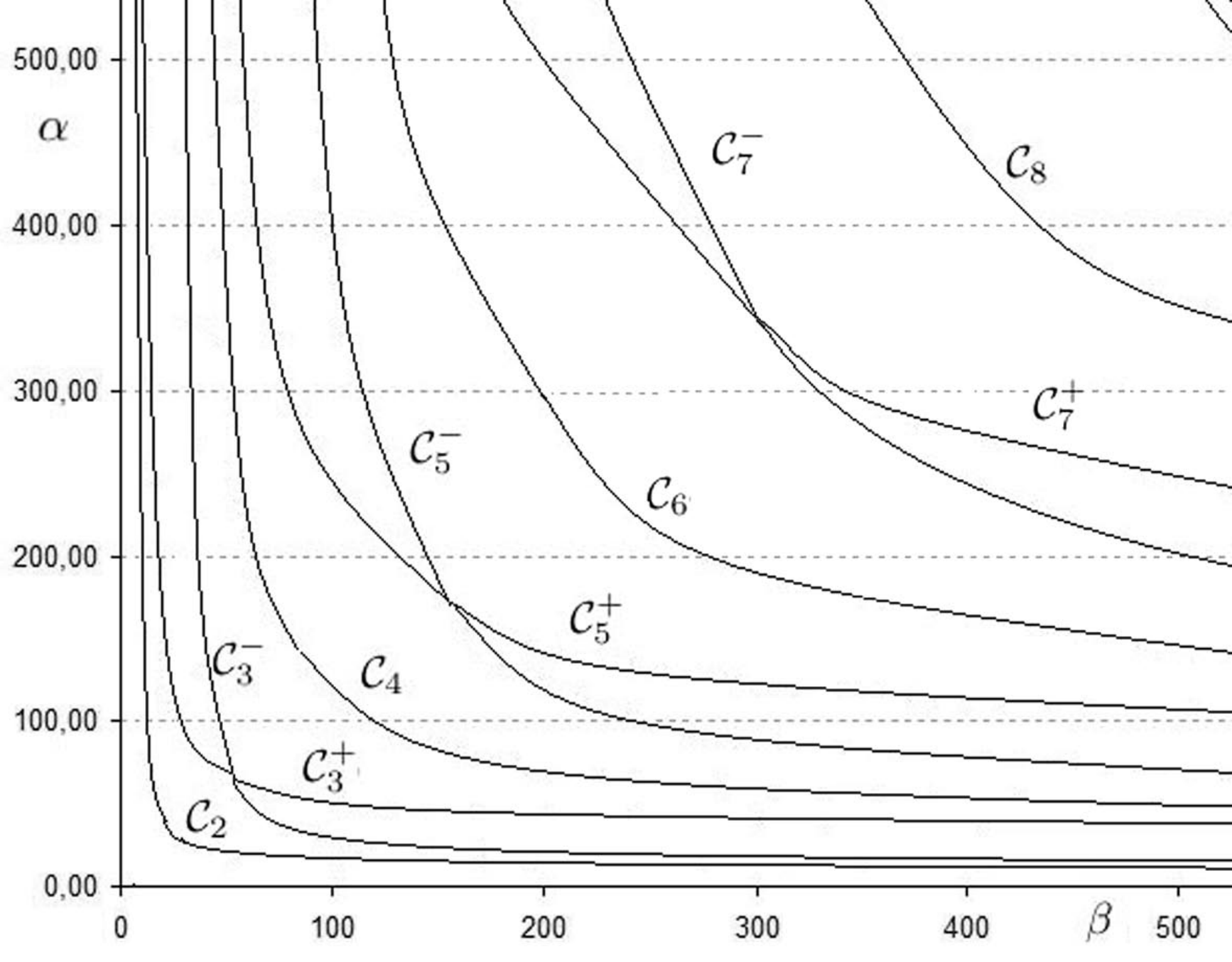}
\vspace*{8pt}
\caption{The Fu{\v{c}}{\'{\i}}k  spectrum for $m(x)=1+(x+1)^{-1}$ and $n(x)=1+cos^2(2x)$.}
\end{figure}

In Table 1  we compute $\lam_{4,t}^+ $ and $t\lam_{4,t}^+$ for different values of $t$ taking
$\ve=0.0001$. For $t\to 0$ and $t\to \infty$ we can compare with the eigenvalues
$\mu_2(m)= 23.44031$, $\mu_2(n)=29.08$.

\begin{table}[th]\label{tabla1}
\caption{Eigenvalues $\alpha=\lam_{4,t}^+ $ and $\beta=t\lam_{4,t}^+$ for different values of $t$.}
\begin{tabular} {c c c}
t  & $\alpha$  & $\beta$  \\ 
$10^5 \qquad $&	23,577&	2357747,078\\
$10^4 \qquad $&	23,939&	239291,613\\
$10^3 \qquad $&	25,110&	25110,064\\
$10^2 \qquad $&	28,994&	2899,356\\
$10  \qquad $&	43,172&	431,716\\
$1  \qquad $&	106,483&	106,483\\
$10^{-1} \qquad $&	486.812&	48,649\\
$10^{-2} \qquad $&	3476.799&	34,768\\
$10^{-3} \qquad $&	30800.052&	30,800\\
$10^{-4} \qquad $&	295937.669&	29,594\\
$10^{-5} \qquad $&	2921329.105&	29,213\\
\end{tabular} 
\end{table}

 In Table 2, for a fixed line with $t=30$, we compute the value of $\lam_{k,t}^+$ for several values of $k$ by using the algorithm
 \eqref{algo} and we compare the obtained values with the asymptotic values obtained in \eqref{asintt}.

\begin{table}[th]
\caption{Higher eigenvalues in a fixed line with $t=30$.}
\begin{tabular}{c c c c }
$k$  & $\lam_{k,t}^+$   & $\lam_{k,t}^+$ & Relative error \\
 & (from \eqref{asintt}) & (from \eqref{algo}) \\ 
10	&211,144 	&212,299	& 0,005 \\
50&	5285,967	&5300,702	 & 0,004 \\
100&	21145,257	& 21132,488	& 0,0006 \\
200	&84503,308 	& 84529,952 &	0,0003 \\
500	&528618,283 & 528312,203	& 0,0006 \\
1000	&2111447,975 &	2113248,815	& 0,0008
\end{tabular} 
\end{table}

Finally, for a  fixed value of $k=28$, we compare in Table 3 for  $\alpha$ the numerical
approximation given by \eqref{algo} and the asymptotic given by \eqref{asintt} varying $t$.

\begin{table}[th]
\caption{Values of $\lam_{28,t}^+$ for different values of $t$.}
\begin{tabular}{c c c c}
$t$  & $\lam_{28,t}^+$   & $\lam_{28,t}^+$ &
 Relative error \\
& (from \eqref{algo}) & (from \eqref{asintt}) \\ 
0,1	 &  23202,100 &  23294,798 &  	0,0039 \\
0,5	 &  7550,103	 &  	7588,970 &  	0,0051 \\
1	 &  5094,391 &  	5124,467 &  	0,0058 \\
5	 &  2565,027  &  2577,485  &  	0,0048 \\
10 &  	2090,991 &  	2099,903  &  	0,0042 \\
1000 &  	1226,496  &  	1231,067 &  	0,0037 \\
100000  &  	1152,512 &  	1156,209 &  	0,0031 \\
\end{tabular} 
\end{table}


\section*{Acknowledgements}

This work was partially supported by Universidad de Buenos Aires under grant UBACYT
20020100100400 and by CONICET (Argentina) PIP 5478/1438.

\end{document}